\documentclass[a4paper,twoside,6pt]{article}

\usepackage[style=authoryear,backend=bibtex8]{biblatex}
\bibliography{reference}

\renewbibmacro*{cite:labelyear+extrayear}{%
  \iffieldundef{labelyear}
    {}
    {\printtext[bibhyperref]{%
       \printfield[parens]{labelyear}%
       \printfield[parens]{extrayear}}}}  % Klammern um Jahr bei Citations 

\DeclareNameFormat{name}{% Bibliographie
\iffirstinits
{\usebibmacro{name:last-first}{#1}{#4}{#5}{#7}}
{\usebibmacro{name:last-first}{#1}{#3}{#5}{#7}}%
\usebibmacro{name:andothers}}

\defbibheading{head}{\section*{References}}

\usepackage[affil-it]{authblk}
\usepackage[top=4.8cm, bottom=6cm, left=3.6cm, right=3.6cm]{geometry} 
\usepackage[english]{babel}
\usepackage[ansinew]{inputenc}
\usepackage{hhline}
\usepackage{framed} 
\usepackage{graphicx}
\usepackage{color}
\usepackage{amsmath}
\usepackage{amssymb}
\usepackage{amsthm} 
\usepackage[geometry]{ifsym}
\usepackage{dsfont}
\usepackage{amsfonts}
\usepackage{accents}

\newtheoremstyle{mystyleplain}
{20pt} % Abstand nach oben
{20pt} % Abstand nach unten
{\itshape} %Schrift des Haupttextes
{}  % indent amount
{\bf} % Schrift des Kopfes
{} % Punktuation nach dem Kopf (kann entweder '.' , ':' oder ' ' sein) 
{\newline} % Abstand nach dem Kopf bis zum Text
{\textbf{\thmname{#1}\thmnumber{ #2}\thmnote{ (#3)}}}
\theoremstyle{mystyleplain}
\newtheorem{proposition}{Proposition}
\newtheorem{lemma}{Lemma}
\newtheorem{theorem}{Theorem}
\newtheorem{corollary}{Corollary}
\renewcommand{\S}{\mathcal S}
\newcommand{\R}{\mathds R}
\newcommand{\N}{\mathds N}
\newcommand{\E}{\mathds E}
\renewcommand{\P}{\mathds P}
\DeclareMathOperator{\Var}{Var}
\numberwithin{equation}{section}

\makeatletter
\def\@maketitle{%
  \newpage
  \vspace*{-\topskip}      % remove the initial space
  \begingroup\centering    % instead of \begin{center}
  \let \footnote \thanks
  \hrule height \z@        % to avoid the insertion of lineskip glue
    {\LARGE \@title \par}%
    \vskip 1.5em 
    {\large
      \lineskip .5em 
      \begin{tabular}[t]{c}%
        \@author
      \end{tabular}\par}%
    \vskip 1em 
    {\large \@date}%
  \par\endgroup            % instead of \end{center}
  \vskip 1.5em             % <--- modify this to adjust the separation
}
\makeatother

\begin{document}
\title{A universal expectation bound on empirical projections of deformed random matrices}
\author{Kamil Jurczak%
		\thanks{E-MAIL: kamil.jurczak@ruhr-uni-bochum.de}}
\affil{Ruhr-Universit\"at Bochum}

\maketitle
\pagestyle{myheadings}

\begin{abstract}
\noindent Let $C$ be a real-valued $M\times M$ matrix with singular values $\lambda_1\ge...\ge\lambda_M$ and $E$ a random matrix of centered i.i.d. entries with finite fourth moment. In this paper we give a universal upper bound on the expectation of $||\hat\pi_rX||_{S_2}^2-||\pi_rX||^2_{S_2}$, where $X:=C+E$ and $\hat\pi_r$ (resp. $\pi_r$) is a rank-$r$ projection maximizing the Hilbert-Schmidt norm $||\tilde\pi_rX||_{S_2}$ (resp. $||\tilde\pi_rC||_{S_2}$) over the set $\S_{M,r}$ of all orthogonal rank-$r$ projections. This result is a generalization of a theorem for Gaussian matrices due to \cite{Rohde2012}. Our approach differs substantially from the techniques of the mentioned article. We analyze $||\hat\pi_rX||_{S_2}^2-||\pi_rX||^2_{S_2}$ from a rather deterministic point of view by an upper bound on $||\hat\pi_rX||_{S_2}^2-||\pi_rX||^2_{S_2}$, whose randomness is totally determined by the largest singular value of $E$.    
\end{abstract}

\section{Introduction}

\thispagestyle{empty}
Let $C$ be a real-valued $M\times M$ matrix, $M\in\N$, with singular values $\lambda_k=\lambda_k(C),~k=1,...,M,$ in decreasing order and $E$ a $M\times M$ random matrix, whose entries are centered i.i.d. real-valued random variables with variance $\sigma^2>0$. We denote the singular values of $E$ by $\sigma_1\ge...\ge\sigma_M$. Further let $\pi_r$ be a rank-$r$ projection, which maximizes the Hilbert-Schmidt norm $||\tilde{\pi}_rC||_{S_2}$ over the set $\S_{M,r}$ of all orthogonal rank-$r$ projections into subspaces of $\R^M$.\\
Consider the process $(Z_{\tilde{\pi}_r})_{\tilde{\pi}_r\in\S_{M,r}}$ defined by
\begin{align}
Z_{\tilde{\pi}_r}:=||\tilde{\pi}_rX||_{S_2}^2-||\pi_rX||_{S_2}^2, ~X:=C+E, \label{eqn:Z}
\end{align}
and its supremum denoted by
\begin{align}
Z_{\hat\pi_r}=\sup_{\tilde{\pi}_r\in \mathcal{S}_{M,r}}Z_{\tilde{\pi}_r}, 
\end{align}
where $\hat\pi_r$ is a location of the supremum. In general, $\hat\pi_r$ is not unique, since the distribution of the entries is allowed to have point masses.\\
\noindent %$Z_{\hat\pi_r}$ plays a crucial in the following problem:   
In statistics one is often not interested to recover the whole matrix $C$ from a measurement $X$ but a low rank approximation containing most of its information. Clearly, for a fixed rank $r$ a ``best'' rank-$r$ approximation is $\pi_r C$ since it minimizes the Hilbert-Schmidt norm $||C-C_r||_{S_2}$ over all $M\times M$ matrices $C_r$ of rank $r$. A natural quantity to find a rank $r$, such that $\pi_rC$ contains sufficient information about $C$, is
\begin{align}
\arg\underset{r\ge 1}{\min}\left\{\frac{||\pi_rC||_{S_2}^2}{||C||_{S_2}^2}\ge\alpha \right\}, \label{eqn:ransel}
\end{align}
where $\alpha\in(0,1]$ is a tuning parameter, which specifies the accuracy of the approximation. The term accuracy is appropriate since for any $r$
\begin{align}||C-\pi_rC||_{S_2}^2=||C||_{S_2}^2-||\pi_rC||_{S_2}^2. \label{eqn:rannor}\end{align}
Clearly, for $\alpha=1$ the expression (\ref{eqn:ransel}) attains the rank of the matrix $C$.
To study quantities like (\ref{eqn:ransel}) or the right-hand side of (\ref{eqn:rannor}) we require an estimate of $||\pi_rC||_{S_2}^2=\sum_{i=1}^r\lambda_i^2$.
Within our model the statistics $||\pi_rX||_{S_2}^2-\sigma^2rM$ is an unbiased estimator for $||\pi_rC||_{S_2}^2$. Since $||\pi_rX||_{S_2}^2-\sigma^2rM$ bases on $\pi_r$, which is typically unknown in advance, naturally the question arises whether the empirical counterpart $||\hat\pi_rX||_{S_2}^2-\sigma^2rM$ is a good alternative estimator. This question may be answered by the study of the expression $\E Z_{\hat\pi_r}$. For a more detailed discussion of the statistical motivation for this problem see \cite{Rohde2012}.\\

\noindent \cite{Rohde2012} investigates the accuracy of empirical reduced-rank projection in case of a Gaussian noise matrix $E$ by upper and lower bounds on $\E Z_{\hat\pi_r}$. The proofs in the mentioned article rely heavily on the Gaussian distribution of $E$. In particular, the main ingredients for the upper bound are among others $S_2$-$S_\infty$-chaining and the \cite{Borell1975} - \cite{Sudakov1974} inequality. Since $Z$ is not centered, the clue of the paper is a slicing argument for $\S_{M,r}$ to proceed to centered Gaussian processes on well-chosen slices. Beyond, for the proofs of lower bounds on $\E Z_{\hat\pi_r}$ the invariance property of the distribution of $E$ under orthogonal transformation and Sudakov's minoration are used. Due to the dependence of the proofs on the Gaussian distribution, naturally the question arises whether the results of \cite{Rohde2012} hold for a larger class of probability distributions of the independent entries $E_{ij}$. Before we pursue this question, we first recapitulate the upper and lower bounds from \cite{Rohde2012}.\\  
\noindent In the following results and the entire article $\lesssim $ means that the left hand side is equal or less than the right  one up to some positive multiplicative constant which does not depend on the variable parameters in the expression. Moreover we denote the projection on the space formed by the first $s$ standard basis vectors of $\R^M$ by $Id_s$.  
\begin{theorem}[Upper bound for Gaussian matrices]
Under the former assumptions and notations let the distribution of $E_{ij}$ be centered Gaussian with variance $\sigma^2$ and $\mathrm{rank}(C)\ge r$. Then in case of $r\le M/2$ the following bound holds
\begin{align}
\E Z_{\hat\pi_r}\lesssim \sigma^2rM\Bigg( \min\Bigg( \frac{\lambda_1^2}{\lambda_r^2},~&1+\frac{\lambda_1}{\sigma\sqrt{M}} \Bigg) \notag \\
&+\min\Bigg(\Bigg(\frac{\frac{1}{r}\sum_{i=r+1}^{2r}\lambda_i^2}{\lambda_r^2}\Bigg)^{\frac{1}{2}}\cdot\frac{\lambda_1}{\sigma\sqrt{M}},~\frac{\lambda_1^2}{\lambda_r^2-\lambda_{r+1}^2} \Bigg) \Bigg),
\end{align}
where $\frac{\lambda_1^2}{\lambda_r^2-\lambda_{r+1}^2}$ is set to infinity, if $\lambda_r=\lambda_{r+1}$.
\end{theorem}
\begin{theorem}[Lower bounds for Gaussian matrices]
Let $E_{ij}$ be centered Gaussian with variance $\sigma^2$.\newline
(i) Let $\lambda_1=...=\lambda_M=\alpha$, then
\begin{align}
\E Z_{\hat\pi_r}\ge \E \Big(\underset{\tilde{\pi}_r\in\S_{M,r}}{\sup}||\tilde\pi_rE||_{S_2}^2-||\pi_rE||_{S_2}^2\Big) \label{eqn:ind1}
\end{align} 
and for $r\le M/2$
\begin{align}
\underset{\alpha\to\infty}{\lim\inf}~\frac{\E Z_{\hat\pi_r}}{\alpha}\gtrsim \sigma r\sqrt{M-r}. \label{eqn:imp}
\end{align}
(ii) Denote 
\begin{align*}
Z^s_{\hat\pi_s}:=\underset{\tilde\pi_s\in\S_{M,r}}{\sup}||\tilde\pi_s \left(C_{\alpha,s}+E\right)||_{S_2}^2-||\pi_s \left(C_{\alpha,s}+E\right)||_{S_2}^2,~1\le s<M,
\end{align*}
where the singular value decomposition of $C_{\alpha,s}$ is given by $U\alpha Id_s V'$, $\alpha>0$. Then it holds
\begin{align}
\underset{\alpha\to\infty}{\lim\inf}\underset{s\in\{r,M-r\}}{\max} \E Z^s_{\hat\pi_s} \gtrsim \sigma^2r(M-r). \label{eqn:ind2}
\end{align}
(iii) Let r=1. There exists an $M_0\in\N$ such that for all $\sigma^2>0$ and any $M\ge M_0$ it holds
\begin{align}
\underset{C\in\R^{M\times M}}{\inf} \E Z_{\hat\pi_r} \gtrsim \E \Big(\underset{\tilde{\pi}_r\in\S_{M,r}}{\sup}||\tilde\pi_rE||_{S_2}^2-||\pi_rE||_{S_2}^2\Big). \label{eqn:ind3}
\end{align}
\end{theorem}
\noindent (\ref{eqn:ind1}), (\ref{eqn:ind2}) and (\ref{eqn:ind3}) indicate that there does not exist a more favorable matrix than $C=0$ in terms of accuracy of $||\hat\pi_rX||_{S_2}$ for $||\pi_rX||_{S_2}$. For $r=1$ this statement is proven. (\ref{eqn:imp}) shows that in general the upper bound $\sigma^2rM(1+\frac{\lambda_1}{\sigma\sqrt{M}})$ is unimprovable. Nevertheless it is possible to state a more refined upper bound, as seen in Theorem 1.\newline 
In this article we generalize Theorem 1 to all random matrices of centered i.i.d. entries with finite fourth moment. Our approach differs significantly from \cite{Rohde2012}. The key argument is an upper bound on $Z_{\hat\pi_r}$, whose randomness is totally determined by $\sigma_1$. This enables us to use an upper bound on the expectation of the spectral radius of a centered random matrix with independent entries by \cite{Latala2005}. \\
%From his result it becomes clear that our generalization of Theorem 1 needs to depend on both the variance and the fourth moment of the entries $E_{ij}$. \\
\noindent In a broad sense we exploit the location $\hat\pi_r$ of the supremum of the process $Z$ to prove the main result of the article. The clue is that $Z$ attains its supremum on a rather small $S_2$-ball depending on $\sigma_1$ for a well-behaved matrix $C$. Our upper bound on $Z_{\hat\pi_r}$ takes this into account.\newline
The main result of this article is the following:
\begin{theorem}[Universal upper bound]
Assume that the entries $E_{ij}$ of the random matrix $E$ have finite variance $\sigma^2$ and finite fourth moment $m_4$. In this case the following inequality holds
\begin{align}
\E Z_{\hat\pi_r}\lesssim r(M-r)\min (\mathord{\mathrm{I}},\mathord{\mathrm{II}},\mathord{\mathrm{III}}),\label{eqn:ErwAbs}
\end{align}
where
\begin{subequations}
\begin{align}
\mathord{\mathrm{I}}&=\sigma^2+\sqrt{m_4}+\frac{\lambda_1}{\sqrt{M}}\left(\sigma+\sqrt[4]{m_4}\right),\\
\mathord{\mathrm{II}}&=\begin{cases}\frac{\lambda_1^2}{\lambda_r^2-\lambda_{r+1}^2}\left(\sigma^2+\sqrt{m_4}\right) &\text{ if } \lambda_r>\lambda_{r+1},\\
 \infty &\text{ if } \lambda_r=\lambda_{r+1},
\end{cases}\\
\mathord{\mathrm{III}}&=\begin{cases}\frac{\lambda_1^2}{\lambda_r^2}\left(\sigma^2+\sqrt{m_4}\right)+\sqrt{\frac{\lambda_1^2\sum_{i=r+1}^{2r}\lambda_i^2}{r(M-r)\lambda_r^2}}\left(\sigma+\sqrt[4]{m_4}\right)
&\text{ if }\lambda_r>0,\\
\infty&\text{ if }\lambda_r=0.
\end{cases}
\end{align}
\end{subequations}
\end{theorem}
\noindent This result is a generalization of Theorem 5.1 of \cite{Rohde2012} (resp. Theorem 1 stated above). We give a brief discussion of this fact later.\newline
The article is structured as follows. In the next section we introduce further notations. We give some elementary estimations on traces of certain matrices in the third section. Most of the results in this section are stated for deterministic matrices. In the fourth section a proof of Theorem 3 is given. Finally in the last section we give a further application of Proposition 1 of section 3. We derive intervals containing $\lim\inf_{M\to\infty} \lambda_1(C_M+E_M)$ and $\lim\sup_{M\to\infty} \lambda_1(C_M+E_M)$ almost surely, where $C_M$ is a deterministic $M\times M$ matrix and $E_M$ is a $M\times M$ random matrix of i.i.d entries with variance $\sigma^2M^{-1}$.

\section{Preliminaries}
We write $\text{tr}(C)$ for the trace of a matrix $C\in\R^{M\times M}$  and $C^T$ for its transpose.
\noindent  In the sequel we split $Z$ into two subprocesses $Z^1$ and $Z^2$ given by
\begin{align*}
Z_{\tilde\pi_r}^1&:=||\tilde\pi_r C||_{S_2}^2-||\pi_r C||_{S_2}^2+2\text{\upshape{tr}}(E^T(\tilde\pi_r-\pi_r )C),\\
Z_{\tilde\pi_r}^2&:=||\tilde\pi_r E||_{S_2}^2-||\pi_r E||_{S_2}^2.
\end{align*}
So it holds $Z=Z^1+Z^2$. Further we denote by $\hat \pi_r^1$ a   location of the supremum of $Z^1$. If $A\lesssim B$ and $B\lesssim A$, we write $A\sim B$. We denote the Schatten-$p$-norm, $1\le p\le\infty$, on $\R^{M\times M}$ by $||\cdot||_{S_p}$. Recall that for $C\in\R^{M\times M}$ with singular values $\lambda_1\ge...\ge\lambda_M$ the Schatten-$p$-norm of $C$ is given by
\begin{align*}
||C||_{S_p}=\sqrt[\scriptstyle p]{\sum_{i=1}^M \lambda_i^p} \text{ for }1\le p\le\infty \text{ and }||C||_{S_\infty}=\lambda_1.
\end{align*}
In particular we will use the Hilbert-Schmidt norm $||\cdot||_{S_2}$ and the spectral norm $||\cdot||_{S_\infty}$. Moreover put $\Delta_r:=\sum_{i=r+1}^{2r}\lambda_i^2$ and $r_M:=r\wedge (M-r)$. The Euclidean sphere in $\R^M$ is denoted by $S^{M-1}$. For any set $B\subset\S_{M,r}$ we define its complement by $B':=\S_{M,r}\setminus B$. Lastly, $\lfloor x \rfloor $ is the largest integer equal or less than $x\in\R$.

\section{Estimation of traces involving differences of projection matrices}
In this section we derive estimations of traces of certain matrices like those arising in the process $Z$. However, the results are stated in a quite general way and are phrased in a deterministic setting.\newline
First recall some basic properties of orthogonal projections. Clearly, we have
\begin{align*}
\pi_r=\pi_r^T \text{ and }\pi_r=\pi_r\pi_r \text{ for }\pi_r\in\S_{M,r}. 
\end{align*}
Therefore every orthogonal projection $\pi_r$ is positive-semidefinite. This implies 
\begin{align*}\text{\upshape{tr}}(\pi^{(1)}_r\pi^{(2)}_r)\ge 0 ~\text{for any}~ \pi^{(1)}_r,\pi^{(2)}_r\in\S_{M,r}. \end{align*}
We conclude
\begin{align*}
||\pi^{(2)}_r-\pi^{(1)}_r||_{S_2}=||(Id-\pi^{(1)}_r)-(Id-\pi^{(2)}_r)||_{S_2}\le\sqrt{2r_M}.
\end{align*}
Finally, note that by symmetry of $\pi^{(2)}_r-\pi^{(1)}_r$ we have
\begin{align*}
||\pi^{(2)}_r-\pi^{(1)}_r||_{S_\infty}&=\underset{x\in S^{M-1}}{\sup}|x^T(\pi^{(2)}_r-\pi^{(1)}_r)x|\\
&=\underset{x\in S^{M-1}}{\sup}|\underbrace{x^T\pi^{(2)}_rx}_{\in [0,1]}- \underbrace{x^T\pi^{(1)}_rx}_{\in [0,1]}|\le 1.
\end{align*}
The next lemma provides a useful estimate to bound $\text{\upshape{tr}}(E^T(\tilde\pi_r-\pi_r )C)$ and $Z_{\tilde\pi_r}^2$. 
\begin{lemma}
Let $\pi_r^{(1)},\pi_r^{(2)}\in\mathcal{S}_{M,r}$ and $A,B\in\R^{M\times M}$, then the following inequality holds
\begin{align}
\text{\upshape{tr}}(A^T(\pi_r^{(2)}-\pi_r^{(1)})B)\le \sqrt{2r_M}||A||_{S_\infty}||B||_{S_\infty}||\pi_r^{(2)}-\pi_r^{(1)}||_{S_2}.
\end{align}
\end{lemma} 
\begin{proof}[Proof]
First, note that \begin{align*}\pi_r^{(2)}-\pi_r^{(1)}=\pi_r^{(2)}-\pi_r^{(2)}\pi_r^{(1)}+\pi_r^{(2)}\pi_r^{(1)}-\pi_r^{(1)}=\pi_r^{(2)}(Id-\pi_r^{(1)})+(\pi_r^{(2)}-Id)\pi_r^{(1)}.\end{align*}
By orthogonality of the decomposition $\pi_r^{(2)}(Id-\pi_r^{(1)})+(\pi_r^{(2)}-Id)\pi_r^{(1)}$ we get
\begin{align}
||(Id-\pi_r^{(2)})\pi_r^{(1)}||_{S_2}=||\pi_r^{(2)}(Id-\pi_r^{(1)})||_{S_2}=\frac{1}{\sqrt{2}}||\pi_r^{(1)}-\pi_r^{(2)}||_{S_2}.
\end{align}
By Cauchy-Schwarz inequality follows
\begin{align*}
\text{\upshape{tr}}(&A^T(\pi_r^{(2)}-\pi_r^{(1)})B)\\
&=\text{\upshape{tr}}(A^T\pi_r^{(2)}(Id-\pi_r^{(1)})B)+\text{\upshape{tr}}(A^T(\pi_r^{(2)}-Id)\pi_r^{(1)}B)\\
&\le \left(||BA^T\pi_r^{(2)}||_{S_2}\wedge ||(Id-\pi_r^{(1)})BA^T||_{S_2}\right)||\pi_r^{(2)}(Id-\pi_r^{(1)})||_{S_2}\\
&\hspace{0.8cm}+\left(||\pi_r^{(1)}BA^T||_{S_2}\wedge ||BA^T(Id-\pi_r^{(2)})||_{S_2}\right)||(\pi_r^{(2)}-Id)\pi_r^{(1)}||_{S_2}\\
&\le \frac{1}{\sqrt{2}}\sqrt{r_M}||BA^T||_{S_\infty}||\pi_r^{(1)}-\pi_r^{(2)}||_{S_2}\\
&\hspace{0.8cm}+\frac{1}{\sqrt{2}}\sqrt{r_M}||BA^T||_{S_\infty}||\pi_r^{(1)}-\pi_r^{(2)}||_{S_2}\\
&\le \sqrt{2r_M}||A||_{S_\infty}||B||_{S_\infty}||\pi_r^{(1)}-\pi_r^{(2)}||_{S_2}.
\end{align*}
\end{proof}
\noindent The statement of the lemma is optimal in the case $M\ge 2r$. The equality attains for matrices 
\begin{align*}\pi_r^{(1)}=\sum\limits_{i=1}^r u_iu_i^T, ~&\pi_r^{(2)}=\sum\limits_{i=1}^r (\sqrt{1-\alpha^2}u_i+\alpha\tilde u_i)(\sqrt{1-\alpha^2}u_i+\alpha \tilde u_i)^T,\\
&A=\mu\text{\upshape{Id}},~B=\nu\left(\pi_r^{(1)}-\pi_r^{(2)}\right),
\end{align*}
where $u_1,...,u_r,\tilde u_1,...,\tilde u_r$ are orthonormal vectors and
 $0\le\alpha\le1,~\mu,\nu>0$. We give a brief computation:
\begin{align*}
\text{\upshape{tr}}(A^T(\pi_r^{(1)}-\pi_r^{(2)})B)&=\mu\nu\text{\upshape{tr}}\Bigg(\left(\pi_r^{(1)}-\sum\limits_{i=1}^r (\sqrt{1-\alpha^2}u_i+\alpha\tilde u_i)(\sqrt{1-\alpha^2}u_i+\alpha \tilde u_i)^T\right)\\
&\hspace{1.1cm}\times\left(\pi_r^{(1)}-\sum\limits_{i=1}^r (\sqrt{1-\alpha^2}u_i+\alpha\tilde u_i)(\sqrt{1-\alpha^2}u_i+\alpha \tilde u_i)^T\right)\Bigg)\\
&=\mu\nu\left(2r-2\text{\upshape{tr}}\left(\pi_r^{(1)}\pi_r^{(2)}\right)\right)\\
&=\mu\nu\left(2r-2r(1-\alpha^2)\right)\\
&=\sqrt{2r}\mu\nu\alpha^2\sqrt{2r}.
\end{align*}
So it remains to show that $||\pi_r^{(1)}-\pi_r^{(2)}||_{S_2}=\alpha\sqrt{2r}$ and $||\pi_r^{(1)}-\pi_r^{(2)}||_{S_\infty}=\alpha$. \\
The first equation is obvious concerning the previous calculation, since
$$||\pi_r^{(1)}-\pi_r^{(2)}||_{S_2}=\sqrt{\text{\upshape{tr}}\left(\left(\pi_r^{(1)}-\pi_r^{(2)}\right)\left(\pi_r^{(1)}-\pi_r^{(2)}\right)\right)}=\alpha\sqrt{2r}.$$
To prove the second equation, one can check that $\alpha$ and $-\alpha$ are the only non-zero eigenvalues of $\pi_r^{(1)}-\pi_r^{(2)}$.
%and their eigenspaces are given by 
%$$W_\alpha=\text{\upshape{span}}\left\{\sqrt{\frac{1+\alpha}{2}}u_i-\sqrt{\frac{1-\alpha}{2}}\tilde u_i\Bigg| i=1,...,r\right\}$$ 
%and
%$$W_{-\alpha}=\text{\upshape{span}}\left\{\sqrt{\frac{1-\alpha}{2}}u_i+\sqrt{\frac{1+\alpha}{2}}\tilde u_i\Bigg| i=1,...,r\right\}.$$
Since $\pi_r^{(1)}-\pi_r^{(2)}$ is symmetric, this implies that $||\pi_r^{(1)}-\pi_r^{(2)}||_{S_\infty}=\alpha.$\newline
\indent As $Z$ has a negative drift, Lemma 1 is not useful to bound $||\tilde\pi_r C||^2_{S_2}-||\pi_rC||^2_{S_2}$. Therefore the next lemma gives an estimate on the drift term. It is significant for our subsequent computations that the distance $||\tilde{\pi}_r-\pi_r||_{S_2}$ influences the drift term rather squared than linearly.
\begin{lemma}~\\
\vspace{-1.1cm}
\begin{itemize}
\item[(i)] For any $\tilde{\pi}_r\in\mathcal{S}_{M,r}$ the following inequality holds
\begin{align}
||\tilde{\pi}_rC||^2_{S_2}-||\pi_rC||^2_{S_2}\le-\frac{1}{2}\left(\lambda_r^2-\lambda_{r+1}^2\right)||\tilde{\pi}_r-\pi_r||^2_{S_2}.
\end{align}
\item[(ii)] For any $\tilde{\pi}_r\in\mathcal{S}_{M,r}$ such that $||\tilde\pi_r-\pi_r||_{S_2}\ge \lambda_r^{-1}\sqrt{2\Delta_r}$, we have
\begin{align}
||\tilde{\pi}_rC||^2_{S_2}-||\pi_rC||^2_{S_2}\le-\frac{1}{2}\lambda_r^2||\tilde\pi_r-\pi_r||_{S_2}^2+\Delta_r.
\end{align}
\end{itemize}
\end{lemma}
\begin{proof}[Proof]
The case $\lambda_r=0$ is trivial. For $\lambda_r>0$ both inequalities follow easily from Proposition 8.1 in \cite{Rohde2012}.
%(i): It is obvious in case of $\lambda_r=\lambda_{r+1}$ or $\tilde\pi_r=\pi_r$, so assume moreover $\lambda_r>\lambda_{r+1}$ and $\tilde\pi_r\neq\pi_r$. Now (i) follows easily from Proposition 8.1 in \cite{Rohde2012} combined with an ``epsilon of room''-argument. Let $\varepsilon>0$ and set 
%\begin{align*}
%||\tilde{\pi}_r-\pi_r||_{S_2}>||\tilde{\pi}_r-\pi_r||_{S_2}-\varepsilon=\sqrt{\frac{1}{\lambda_r^2-\lambda_{r+1}^2}}\sqrt{2\delta}.
%\end{align*}
%Then from statement (8.1) in the mentioned proposition we get
%\begin{align*}
%||\pi_rC||_{S_2}^2-||\tilde{\pi}_rC||_{S_2}^2>\delta=\frac{1}{2}\left(\lambda_r^2-\lambda_{r+1}^2\right)(||\pi_r-\tilde{\pi}_r||_{S_2}-\varepsilon)^2.
%\end{align*}
%By multiplication by -1 and taking the limit $\varepsilon\to 0$ (i) is proved.\\
%(ii): This part follows analogous to (i).
\end{proof}
\noindent Now we derive an upper bound on $Z^1_{\hat\pi_r^1}$, which will be useful to estimate the expectation of $Z_{\hat\pi_r}$. In a certain way the upper bound regards the location of $\hat\pi_r^1$.
\begin{proposition}
For the supremum $Z^1_{\hat{\pi}^1_r}$ of the process $Z^1$ we have $Z^1_{\hat{\pi}^1_r}\le Y$ with
\begin{align}
Y:=\min\left(\mathord{\mathrm{I}}',\mathord{\mathrm{II}}',\mathord{\mathrm{III}}'\right),
\end{align}
where
\begin{subequations}
\begin{align}
\mathord{\mathrm{I}}'&:=4r_M\lambda_1\sigma_1,\\
\mathord{\mathrm{II}}'&:=\begin{cases}4r_M\frac{\lambda_1^2}{\lambda_r^2-\lambda_{r+1}^2}\sigma_1^2 &\text{ if }\lambda_r>\lambda_{r+1},\\
\infty&\text{ if }\lambda_r=\lambda_{r+1},
\end{cases}\\
\mathord{\mathrm{III}}'&:=\begin{cases}\max\left(4\sqrt{r_M\Delta_r}\frac{\lambda_1}{\lambda_r}\sigma_1,~8r_M\frac{\lambda_1^2}{\lambda_r^2}\sigma_1^2\right)&\text{ if }\lambda_r>0,\\
\infty&\text{ if }\lambda_r=0.
\end{cases}
\end{align}
\end{subequations}
\end{proposition}
\begin{proof}
We prove for $\mathord{\mathrm{I}}',\mathord{\mathrm{II}}'$ and $\mathord{\mathrm{III}}'$ that $Z^1_{\hat\pi_r^1}$ is less or equal to each one.\\
$Z^1_{\hat\pi_r^1}\le\mathord{\mathrm{I}}':$ Since $||\tilde\pi_rC||_{S_2}^2-||\pi_rC||_{S_2}^2\le 0$, we get by Lemma 1 for any $\tilde{\pi}_r\in\mathcal{S}_{M,r}$
\begin{align}
Z^1_{\tilde\pi_r}\le 2\sqrt{2r_M}\sigma_1\lambda_1||\tilde\pi_r-\pi_r||_{S_2}\le 4r_M\sigma_1\lambda_1.
\end{align}
%So $Z^1_{\hat\pi_r^1}\le\mathord{\mathrm{I}}'$ holds.\\
$Z^1_{\hat\pi_r^1}\le\mathord{\mathrm{II}}':$ Assume $\lambda_r>\lambda_{r+1}$. 
We obtain by Lemma 1 and 2(i) for any $\tilde{\pi}_r\in\mathcal{S}_{M,r}$
\begin{align*}
Z_{\tilde\pi_r}^1&=||\tilde{\pi}_rC||_{S_2}^2-||\pi_rC||_{S_2}^2+2\text{\upshape{tr}}(E^T(\tilde{\pi}_r-\pi_r )C)\\
&\le -\frac{1}{2} \left(\lambda_r^2-\lambda_{r+1}^2\right)||\tilde{\pi}_r-\pi_r||_{S_2}^2+2\sqrt{2r_M}\sigma_1\lambda_1||\tilde{\pi}_r-\pi_r||_{S_2}\\
\end{align*}
Then maximizing the right-hand side of the inequality
\begin{align*}
Z_{\tilde\pi_r}^1\le -\frac{1}{2}\left(\lambda_r^2-\lambda_{r+1}^2\right)||\tilde{\pi}_r-\pi_r||_{S_2}^2+2\sqrt{2r_M}\sigma_1\lambda_1||\tilde{\pi}_r-\pi_r||_{S_2}
\end{align*}
over all $x:=||\tilde\pi_r-\pi_r||_{S_2}$ provides the claim.\\
$Z^1_{\hat\pi_r^1}\le\mathord{\mathrm{III}}':$ Assume $\lambda_r>0$. In order to prove the last bound we split $S_{M,r}$ into two sets and take the supremum of $Z^1$ on this sets separately. We define
\begin{align}
B_{\mathord{\mathrm{III}}'}:=\{\tilde\pi_r\in\S_{M,r}:||\tilde\pi_r-\pi_r||_{S_2}<\lambda_r^{-1}\sqrt{2\Delta_r}\}.
\end{align}
It holds
\begin{align}
Z^1_{\hat\pi_r^1}=\max\left(\underset{\tilde\pi_r\in B_{\mathord{\mathrm{III}}'}}{\sup}Z^1_{\tilde\pi_r},~\underset{\tilde\pi_r\in B'_{\mathord{\mathrm{III}}'}}{\sup}Z^1_{\tilde\pi_r}\right).\label{eqn:sepSup}
\end{align}
For the first expression in the maximum of (\ref{eqn:sepSup}) we get analogous to the proof of $Z^1_{\hat\pi_r^1}\le\mathord{\mathrm{I}}'$:
\begin{align}
\underset{\tilde\pi_r\in B_{\mathord{\mathrm{III}}'}}{\sup}Z^1_{\tilde\pi_r}\le \underset{\tilde\pi_r\in B_{\mathord{\mathrm{III}}'}}{\sup} 2\sqrt{2r_M}\sigma_1\lambda_1||\tilde\pi_r-\pi_r||_{S_2}\le 4\sqrt{r_M\Delta_r}\frac{\lambda_1}{\lambda_r}\sigma_1. \label{eqn:ersSup}
\end{align}
It remains to bound the second expression in the maximum of (\ref{eqn:sepSup}). By Lemma 1 again and by Lemma 2(ii) follows for any $\tilde\pi_r\in B'_{\mathord{\mathrm{III}}'}$
\begin{align}
Z^1_{\tilde\pi_r}\le-\frac{1}{2}\lambda_r^2||\tilde\pi_r-\pi_r||_{S_2}^2+\Delta_r+2\sqrt{2r_M}\sigma_1\lambda_1||\tilde\pi_r-\pi_r||_{S_2}.\label{eqn:nerEqn}
\end{align}
The right-hand side attains its global maximum on 
\begin{align}
\{\tilde\pi_r\in\S_{M,r}:||\tilde\pi_r-\pi_r||_{S_2}=2\sqrt{2r_M}\sigma_1\frac{\lambda_1}{\lambda_r^2}\wedge \sqrt{2r_M}\}. \label{eqn:locMax}
\end{align}
If $2\sqrt{2r_M}\sigma_1\frac{\lambda_1}{\lambda_r}<\sqrt{2\Delta_r}$, then it holds 
\begin{align*}
\{\tilde\pi_r\in\S_{M,r}:||\tilde\pi_r-\pi_r||_{S_2}=2\sqrt{2r_M}\sigma_1\frac{\lambda_1}{\lambda_r^2}\wedge \sqrt{2r_M}\}\cap B'_{\mathord{\mathrm{III}}'}=\emptyset.
\end{align*}
In this case due to reasons of monotonicity the right-hand side of (\ref{eqn:nerEqn}) restricted to $\{\tilde\pi_r\in\S_{M,r}:||\tilde\pi_r-\pi_r||_{S_2}\ge\lambda_r^{-1}\sqrt{2\Delta_r}\}$ attains its minimum on $\{\tilde\pi_r\in\S_{M,r}:||\tilde\pi_r-\pi_r||_{S_2}=\lambda_r^{-1}\sqrt{2\Delta_r}\}$. So we have
\begin{align}
\mathds{1}_{\{\sqrt{2\Delta_r}>2\sqrt{2r_M}\frac{\lambda_1}{\lambda_r}\sigma_1\}}Z^1_{\tilde\pi_r}\le\mathds{1}_{\{\sqrt{2\Delta_r}>2\sqrt{2r_M}\frac{\lambda_1}{\lambda_r}\sigma_1\}} 4\sqrt{r_M\Delta_r}\frac{\lambda_1}{\lambda_r}\sigma_1. \label{eqn:zweSup1}
\end{align}
Otherwise by (\ref{eqn:locMax}) follows
\begin{align}
\mathds{1}_{\{\sqrt{2\Delta_r}\le 2\sqrt{2r_M}\frac{\lambda_1}{\lambda_r}\sigma_1\}}Z^1_{\tilde\pi_r} &\le \mathds{1}_{\{\sqrt{2\Delta_r}\le 2\sqrt{2r_M}\frac{\lambda_1}{\lambda_r}\sigma_1\}}\left(4r_M\frac{\lambda_1^2}{\lambda_r^2}\sigma_1^2+\Delta_r\right) \notag \\
&\le \mathds{1}_{\{\sqrt{2\Delta_r}\le 2\sqrt{2r_M}\frac{\lambda_1}{\lambda_r}\sigma_1\}}8r_M\frac{\lambda_1^2}{\lambda_r^2}\sigma_1^2.
 \label{eqn:zweSup2}
\end{align}
By (\ref{eqn:zweSup1}) and (\ref{eqn:zweSup2}) we get
\begin{align}
\underset{\tilde\pi_r\in B'_{\mathord{\mathrm{III}}'}}{\sup}Z^1_{\tilde\pi_r}&=\underset{\tilde\pi_r\in B'_{\mathord{\mathrm{III}}'}}{\sup}~\mathds{1}_{\{\sqrt{2\Delta_r}>2\sqrt{2r_M}\frac{\lambda_1}{\lambda_r}\sigma_1\}}Z^1_{\tilde\pi_r}+\mathds{1}_{\{\sqrt{2\Delta_r}\le 2\sqrt{2r_M}\frac{\lambda_1}{\lambda_r}\sigma_1\}}Z^1_{\tilde\pi_r} \notag \\
&\le \mathds{1}_{\{\sqrt{2\Delta_r}>2\sqrt{2r_M}\frac{\lambda_1}{\lambda_r}\sigma_1\}} 4\sqrt{r_M\Delta_r}\frac{\lambda_1}{\lambda_r}\sigma_1+\mathds{1}_{\{\sqrt{2\Delta_r}\le 2\sqrt{2r_M}\frac{\lambda_1}{\lambda_r}\sigma_1\}}8r_M\frac{\lambda_1^2}{\lambda_r^2}\sigma_1^2 \notag \\
%&\le \underset{\tilde\pi_r\in B'_{\mathord{\mathrm{III}}'}}{\sup}\max\left(4\sqrt{r\Delta_r}\frac{\lambda_1}{\lambda_r}\sigma_1,~8r\frac{\lambda_1^2}{\lambda_r^2}\sigma_1^2\right) \notag \\
&\le \max\left(4\sqrt{r_M\Delta_r}\frac{\lambda_1}{\lambda_r}\sigma_1,~8r_M\frac{\lambda_1^2}{\lambda_r^2}\sigma_1^2\right). \label{eqn:zweSup}
\end{align}
Finally combining (\ref{eqn:ersSup}) and (\ref{eqn:zweSup}) yields $Z^1_{\hat\pi_r^1}\le\mathord{\mathrm{III}}'$.
\end{proof}

\section{Proof of Theorem 3}
\noindent Now we are ready to prove Theorem 3. We first state a result by \cite{Latala2005} in simplified terms. 
\begin{theorem}[\cite{Latala2005}]
For any $M\times M$ random matrix $E$ of centered i.i.d. entries with variance $\sigma^2$ and fourth moment $m_4$ the following inequality holds
\begin{align}
\E\sigma_1^2\lesssim M\left(\sigma^2+\sqrt{m_4}\right). \label{eqn:lat}
\end{align}
\end{theorem}
\noindent Note that the original result is phrased for the expectation of $\sigma_1$ and not of $\sigma_1^2$, but actually the proof includes statement (\ref{eqn:lat}).
\begin{proof}[Proof of Theorem 3.] Beforehand note that by distinguishing the cases $r<\frac{M}{2}$ and $r\ge\frac{M}{2}$ it holds
\begin{align*}
r(M-r)\le r_MM\le 2 r(M-r).
\end{align*}
Now we commence the proof of Theorem 3. We get
\begin{align}
\E Z_{\hat\pi_r}\le\E Z^1_{\hat\pi_r^1}+ \E \underset{\tilde\pi_r\in\S_{M,r}}{\sup}Z^2_{\tilde\pi_r}.\label{eqn:hilthe}
\end{align}
We first consider the second summand in (\ref{eqn:hilthe}). By Lemma 1 we have
\begin{align}
\E \underset{\tilde\pi_r\in\S_{M,r}}{\sup}Z^2_{\tilde\pi_r}&=\E \underset{\tilde\pi_r\in\S_{M,r}}{\sup}\text{tr}\left(E^T(\tilde\pi_r-\pi_r)E\right) \notag \\
&\le \sqrt{2r_M} \underset{\tilde\pi_r\in\S_{M,r}}{\sup}||\tilde\pi_r-\pi_r||_{S_2}  \E \sigma_1^2 \label{eqn:hilthe2}
\end{align}
Recall that $||\tilde\pi_r-\pi_r||_{S_2}\le \sqrt{2r_M}$ for any $\tilde\pi_r\in\S_{M,r}$ and 
apply Theorem 4 to (\ref{eqn:hilthe2})
\begin{align}
\E \underset{\tilde\pi_r\in\S_{M,r}}{\sup}Z^2_{\tilde\pi_r}\le 2r_M \E\sigma_1^2&\lesssim r_MM\left(\sigma^2+\sqrt{m_4}\right) \notag\\
&\lesssim r(M-r)\left(\sigma^2+\sqrt{m_4}\right).
\end{align}
Therefore
\begin{align}
\E \underset{\tilde\pi_r\in\S_{M,r}}{\sup}Z^2_{\tilde\pi_r}\lesssim r (M-r) \min (\mathord{\mathrm{I}},~\mathord{\mathrm{II}},~\mathord{\mathrm{III}}).
\end{align}
So it remains to prove that
\begin{align}
\E Z^1_{\hat\pi_r^1}\lesssim r (M-r) \min (\mathord{\mathrm{I}},~\mathord{\mathrm{II}},~\mathord{\mathrm{III}}).
\end{align}
By Proposition 1, monotonicity of integral and Theorem 4 we get
\begin{align*}
\E Z^1_{\hat\pi_r^1}&\le \E Y \\
&\le \min\left(\E\mathord{\mathrm{I}}',~\E\mathord{\mathrm{II}}',~\E\mathord{\mathrm{III}}'\right)\\
&\le \min\left(4r_M\lambda_1\E\sigma_1,~4r_M\frac{\lambda_1^2}{\lambda_r^2-\lambda_{r+1}^2}\E\sigma_1^2,~ \E\max\left(4\sqrt{r_M\Delta_r}\frac{\lambda_1}{\lambda_r}\sigma_1,~8r_M\frac{\lambda_1^2}{\lambda_r^2}\sigma_1^2\right) \right)\\
&\le \min\left(4r_M\lambda_1\E\sigma_1,~ 4r_M\frac{\lambda_1^2}{\lambda_r^2-\lambda_{r+1}^2}\E\sigma_1^2,~ 4\sqrt{r_M\Delta_r}\frac{\lambda_1}{\lambda_r}\E\sigma_1+8r_M\frac{\lambda_1^2}{\lambda_r^2}\E\sigma_1^2 \right)\\
&\lesssim r(M-r) \min\Bigg(\frac{\lambda_1}{\sqrt{M}}\left(\sigma+\sqrt[4]{m_4}\right),~\frac{\lambda_1^2}{\lambda_r^2-\lambda_{r+1}^2}\left(\sigma^2+\sqrt[2]{m_4}\right),\\
&\hspace{3.6cm}~\frac{\lambda_1^2}{\lambda_r^2}\left(\sigma^2+\sqrt{m_4}\right)+\sqrt{\frac{\lambda_1^2\sum_{i=r+1}^{2r}\lambda_i^2}{r(M-r)\lambda_r^2}}\left(\sigma+\sqrt[4]{m_4}\right)\Bigg)\\
&\lesssim r (M-r) \min \left(\mathord{\mathrm{I}},~\mathord{\mathrm{II}},~\mathord{\mathrm{III}}\right).
\end{align*}
\end{proof}
\noindent As mentioned in the introduction, this result is a generalization of Theorem 5.1 of \cite{Rohde2012}. To check this consider the case, where $E$ is a Gaussian matrix and $r\le M/2$. Since the fourth moment of a centered Gaussian random variable is given by $3\sigma^4$, the right-hand side of inequality (\ref{eqn:ErwAbs}) may be rewritten as
\begin{align*}
\sigma^2rM\min\left(1+\frac{\lambda_1}{\sigma\sqrt{M}},~\frac{\lambda_1^2}{\lambda_r^2-\lambda_{r+1}^2},~\frac{\lambda_1^2}{\lambda_r^2}+\sqrt{\frac{\tfrac{1}{r}\sum_{i=r+1}^{2r}\lambda_i^2}{\lambda_r^2}}\frac{\lambda_1}{\sigma\sqrt{M}}\right),
\end{align*} 
where the constant in (\ref{eqn:ErwAbs}) is now specific to Gaussian matrices. So we have to show that
\begin{align*}
&\min\left(\frac{\lambda_1^2}{\lambda_r^2},~1+\frac{\lambda_1}{\sigma\sqrt{M}}\right)+\min\left(\sqrt{\frac{\tfrac{1}{r}\sum_{i=r+1}^{2r}\lambda_i^2}{\lambda_r^2}}\frac{\lambda_1}{\sigma\sqrt{M}},~\frac{\lambda_1^2}{\lambda_r^2-\lambda_{r+1}^2}\right)\\
&\hspace{4cm}\sim \min\left(1+\frac{\lambda_1}{\sigma\sqrt{M}},~\frac{\lambda_1^2}{\lambda_r^2-\lambda_{r+1}^2},~\frac{\lambda_1^2}{\lambda_r^2}+\sqrt{\frac{\tfrac{1}{r}\sum_{i=r+1}^{2r}\lambda_i^2}{\lambda_r^2}}\frac{\lambda_1}{\sigma\sqrt{M}}\right).
\end{align*}
This follows by (\ref{eqn:proEqi1}) and (\ref{eqn:proEqi2}) in the next computation
\begin{align}
&\hspace{-0.3cm}\min\left(\frac{\lambda_1^2}{\lambda_r^2},~1+\frac{\lambda_1}{\sigma\sqrt{M}}\right)+\min\left(\sqrt{\frac{\tfrac{1}{r}\sum_{i=r+1}^{2r}\lambda_i^2}{\lambda_r^2}}\frac{\lambda_1}{\sigma\sqrt{M}},~\frac{\lambda_1^2}{\lambda_r^2-\lambda_{r+1}^2}\right) \notag  \\ 
&\le \min\left(1+\frac{\lambda_1}{\sigma\sqrt{M}}+\sqrt{\frac{\tfrac{1}{r}\sum_{i=r+1}^{2r}\lambda_i^2}{\lambda_r^2}}\frac{\lambda_1}{\sigma\sqrt{M}},
%\right.\notag \\
%&\hspace{4.7cm}\left.
~\frac{\lambda_1^2}{\lambda_r^2}+\frac{\lambda_1^2}{\lambda_r^2-\lambda_{r+1}^2},~\frac{\lambda_1^2}{\lambda_r^2}+\sqrt{\frac{\tfrac{1}{r}\sum_{i=r+1}^{2r}\lambda_i^2}{\lambda_r^2}}\frac{\lambda_1}{\sigma\sqrt{M}}\right) \notag \\
&\le 2 \min\left(1+\frac{\lambda_1}{\sigma\sqrt{M}},~\frac{\lambda_1^2}{\lambda_r^2-\lambda_{r+1}^2},~\frac{\lambda_1^2}{\lambda_r^2}+\sqrt{\frac{\tfrac{1}{r}\sum_{i=r+1}^{2r}\lambda_i^2}{\lambda_r^2}}\frac{\lambda_1}{\sigma\sqrt{M}}\right) \label{eqn:proEqi1}\\
&\le 2 \min\left(1+\frac{\lambda_1}{\sigma\sqrt{M}},~\frac{\lambda_1^2}{\lambda_r^2}+\frac{\lambda_1^2}{\lambda_r^2-\lambda_{r+1}^2},~\frac{\lambda_1^2}{\lambda_r^2}+\sqrt{\frac{\tfrac{1}{r}\sum_{i=r+1}^{2r}\lambda_i^2}{\lambda_r^2}}\frac{\lambda_1}{\sigma\sqrt{M}}\right) \notag \\
&=  2\min\left(1+\frac{\lambda_1}{\sigma\sqrt{M}},~\frac{\lambda_1^2}{\lambda_r^2}+\min\left(\frac{\lambda_1^2}{\lambda_r^2-\lambda_{r+1}^2},~\sqrt{\frac{\tfrac{1}{r}\sum_{i=r+1}^{2r}\lambda_i^2}{\lambda_r^2}}\frac{\lambda_1}{\sigma\sqrt{M}}\right)\right) \notag\\
&\le 2\left(\min\left(\frac{\lambda_1^2}{\lambda_r^2},~1+\frac{\lambda_1}{\sigma\sqrt{M}}\right)
%\right.
%\notag \\
%&\hspace{3.5cm}\left.
+\min\left(\sqrt{\frac{\tfrac{1}{r}\sum_{i=r+1}^{2r}\lambda_i^2}{\lambda_r^2}}\frac{\lambda_1}{\sigma\sqrt{M}},~\frac{\lambda_1^2}{\lambda_r^2-\lambda_{r+1}^2}\right)\right). \label{eqn:proEqi2}
\end{align}
The last line arises by the simple observation that $\min(a,~b+c)\le \min(a,b)+c$ for any $a,b\in\R$ and $c\ge0$.

\section{Application: Localizing the largest singular value of a deformed random matrix}

\noindent As a further application of Proposition 1 we take a classical view on random matrices. Hence, let $(E_{ij})_{i,j\in\N}$ be a doubly indexed sequence of centered i.i.d. random variables with variance $\sigma^2$ and finite fourth moment. By $(E_M)$ we denote the sequence of $M\times M$ random matrices $1/\sqrt{M}(E_{ij})_{i,j\le M}$. $(C_M)$ is a sequence of deterministic $M\times M$ matrices.  Assume furthermore that the first and the second singular values $\lambda_1(C_M)$ and $\lambda_2(C_M)$ converge to some real numbers $\lambda_1>\lambda_2\ge0$ for $M\to\infty$. We specify an interval containing ${\lim\inf}_{M\to\infty}~\lambda_1(C_M+E_M)$ and ${\lim\sup}_{M\to\infty}~\lambda_1(C_M+E_M)$ almost surely.
\begin{corollary} 
Under the former notations and assumptions let $(u_{i1})_{i\in\N}$ be a sequence of real numbers such that 
\begin{align*}\left(\sum_{i=1}^M u_{i1}^2\right)^{-\frac{1}{2}}(u_{11},...,u_{M1})^T\end{align*} 
is the left singular vector of $C_M$ corresponding to the largest singular value. If there exist $\beta>1,\beta'>0$ and a constant $c>0$, which depends only on $\beta$ and $\beta'$, such that
\begin{align}
\frac{\sum_{i=B}^M u_{i1}^2}{\sum_{i=1}^Mu_{i1}^2}\le  \frac{c}{M^{\beta'}}\text{ for all } M\in\N,~\text{where }B=\lfloor(M^\frac{1}{\beta}-1)^{\beta}\rfloor, \label{eqn:SLC}
\end{align}
then it holds a.s.
\begin{align}
\sqrt{\lambda_1^2+\sigma^2}&\le\underset{M\to\infty}{\lim\inf}~\lambda_1(C^M+E^M)\notag\\
&\le\underset{M\to\infty}{\lim\sup}~\lambda_1(C^M+E^M)\le \sqrt{\lambda_1^2+4\sigma^2+16\sigma^2\frac{\lambda_1^2}{\lambda_1^2-\lambda_2^2}}.\label{eqn:int}
\end{align}
\end{corollary}

\noindent Thus, if in the large amplitude regime the values $\lambda_1$ and $\lambda_2$ are well-separated, then the largest singular value of $C_M+E_M$ is typically close to $\lambda_1$ but larger. This result can be seen complementary to \cite{Georges2012}. They consider finite rank perturbations of a sequence of random matrices $(X_M)$, where $X_M$ is a $M\times N$-matrix. Under certain assumptions they show an almost sure convergence of the largest singular values in the limit $M,N_M\to\infty$. Since we only make assumptions on the first two singular values and the first left singular vector of the perturbation matrices $(C_M)$, the limit $\lim_{M\to\infty}~\lambda_1(C_M+E_M)$ does not exist in general. \\%However, if $\lambda_1-\lambda_2$ is much larger than $10\sigma^2$, there exists a small interval containing $\lim\inf_{M\to\infty}~\lambda_1(C_M+E_M)$ and $\lim\sup_{M\to\infty}~\lambda_1(C_M+E_M)$ almost surely. 
Note that if $\lambda_1-\lambda_2\ge 4\sigma$, then the upper bound in Corollary 1 is already better than the bound $\lambda_1+2\sigma$ on $\lim\sup_{M\to\infty}~\lambda_1(C_M+E_M)$. The inequality  $\lim\sup_{M\to\infty}~\lambda_1(C_M+E_M)\le\lambda_1+2\sigma$ holds without any additional structural assumptions on $(C_M)$, since we may use the triangle inequality on the spectral norm and the well-known result by \cite{Krishnaiah} that $\lim_{M\to\infty} \sigma_1= 2\sigma$ a.s.\\
\noindent Before we prove Corollary 1, let us give two examples of sequences $(u_i)_{i\in\N}$ satisfying condition (\ref{eqn:SLC}):
\begin{itemize}
\item All but finitely many $u_i$'s are zero.
\item The sequence is bounded and bounded away from zero. 
\end{itemize}
\begin{proof}[Proof of Corollary 1.] 
Now we give a computation of (\ref{eqn:int}). For this purpose we make a slight abuse of notations. We write $C$ and $E$ for the matrices $C_M$ and $E_M$. Further let $(v_{11},...,v_{M1})$ be the right singular vector of $C_M$ corresponding to $\lambda_1(C_M)$. Consider the lower bound on $\underset{M\to\infty}{\lim\inf}~\lambda_1(C_M+E_M)$: 
\begin{align*}
\lambda_1(C+E)&\ge ||\pi_1(C+E)||_{S_2}= \sqrt{\lambda_1(C)^2+2\text{tr}(C^T\pi_1E)+||\pi_1E||_{S_2}^2}\\
&\hspace{0cm}=\sqrt{\lambda_1(C)^2+2\lambda_1(C)\sum\limits_{i,j=1}^M\frac{v_{i1}u_{j1}}{(\sum_{i=1}^M u_{i1}^2)^{1/2}}E_{ji}+\sum\limits_{i,j,k=1}^M\frac{ u_{i1}u_{j1}}{\sum_{i=1}^M u_{i1}^2}E_{ik}E_{jk}}.
\end{align*}
We use a strong law of large numbers given by Theorem 3 of \cite{Thrum1987}  to get 
\begin{align}\sum\limits_{i,j=1}^M\frac{v_{i1}u_{j1}}{(\sum_{i=1}^M u_{i1}^2)^{1/2}}E_{ji}\overset{a.s.}{\to}0.\label{eqn:hilcor2}\end{align}
Let us check that the left hand side of (\ref{eqn:hilcor2}) actually fulfills the assumptions of Thrum's strong law of large number. Therefore identify the objects therein as follows
\begin{align*}
n:=M^2,~a_{n,i,j}:=\frac{v_{i1}u_{j1}}{(\sum_{i=1}^M u_{i1}^2)^{1/2}}\text{ and }X_{i,j}:=\sqrt{M}E_{ji}.
\end{align*}
Note that we keep the double index. Clearly, the first four moments of $X_{i,j}$ exist and $\sum_{i,j=1}^M a_{n,i,j}^2=1$. Therefore
\begin{align*}
\sum\limits_{i,j=1}^M\frac{v_{i1}u_{j1}}{(\sum_{i=1}^M u_{i1}^2)^{1/2}}E_{ji}=\sum\limits_{i,j=1}^Ma_{n,i,j}X_{i,j}n^{-1/4}\overset{a.s.}{\to}0.
\end{align*}

\noindent Moreover assumption (\ref{eqn:SLC}) allows to use the subsequent Theorem 5 to obtain
\begin{align}\sum\limits_{i,j,k=1}^M\frac{ u_{i1}u_{j1}}{\sum_{i=1}^M u_{i1}^2}E_{ik}E_{jk}\overset{a.s.}{\to}\sigma^2. \label{eqn:hilcor}\end{align}

\noindent We conclude that
\begin{align*}
\sqrt{\lambda_1^2+\sigma^2}\le\underset{M\to\infty}{\lim\inf}~\lambda_1(C_M+E_M)\text{ a.s.}
\end{align*}
It remains to prove the upper bound. Using Proposition 1, one gets
\allowdisplaybreaks{\begin{align*}
\lambda_1(C&+E)=\sqrt{||\hat\pi_1(C+E)||_{S_2}^2-||\pi_1(C+E)||_{S_2}^2+||\pi_1(C+E)||_{S_2}^2}\\
&\le \sqrt{\mathord{\mathrm{II}}'+\sigma_1^2+\lambda_1(C)^2+2\lambda_1(C)\sum\limits_{i,j=1}^M\frac{v_{i1}u_{j1}}{(\sum_{i=1}^M u_{i1}^2)^{1/2}}E_{ji}}\\
&= \sqrt{4\frac{\lambda_1(C)^2}{\lambda_1(C)^2-\lambda_{2}(C)^2}\sigma_1^2+\sigma_1^2+\lambda_1(C)^2+2\lambda_1(C)\sum\limits_{i,j=1}^M \frac{v_{i1}u_{j1}}{(\sum_{i=1}^M u_{i1}^2)^{1/2}}E_{ji}}.
\end{align*}}

\noindent By results of \cite{Krishnaiah} and \cite{Silverstein} we know that the fourth-moment condition on the entries of $E$ is necessary and sufficient for the almost sure convergence of $\sigma_1$ to $2\sigma$. Applying this to the last line of the computation yields the desired claim. 
\end{proof}
\noindent Note that the assumption (\ref{eqn:SLC}) on the first singular vector of $C_M$ is only needed for the lower bound in Corollary 1.\\
\noindent We close this article by a strong law of large numbers for empirical covariance matrices. In this result the empirical covariance matrix is considered as a quadratic form.
\begin{theorem}[SLLN for empirical covariance matrices]
Let $(u_i)_{i\in\N}$ be a sequence of real numbers such that there exist $\beta>1,\beta'>0$ and a constant $c>0$, which may depend on $\beta$ and $\beta'$, with
\begin{align}
\frac{\sum_{i=B}^M u_{i}^2}{\sum_{i=1}^Mu_{i}^2}\le \frac{c}{M^{\beta'}}\text{ for all } M\in\N,~\text{where }B=\lfloor(M^\frac{1}{\beta}-1)^{\beta}\rfloor. \label{eqn:SLL2}
\end{align}
Furthermore let $(E_{ij})_{i,j\in\N}$ be a doubly indexed sequence of centered i.i.d. random variables with variance $\sigma^2$ and finite fourth moment. By $(E_M)$ we denote the sequence of $M\times M$ random matrices $1/\sqrt{M}(E_{ij})_{i,j\le M}$. Then we have
\begin{align}
 Z_M:=\tilde u_M^TE_ME_M^T\tilde u_M\overset{a.s.}{\to}\sigma^2,
\end{align}
where 
\begin{align*}
\tilde u_M:=\left(\sum_{i=1}^Mu_{1}^2\right)^{-\frac{1}{2}}(u_1,...,u_M)^T.
\end{align*}
\end{theorem}
\begin{proof}
Rewrite
\begin{align*}
Z_M=\frac{1}{M\sum_{i=1}^Mu_{i}^2}\sum\limits_{k=1}^M\sum\limits_{i,j=1}^Mu_{i}u_jE_{ik}E_{jk}.
\end{align*}
Therefore $Z_M$ is the sum of $W_M$ and $2X_M$ given by
\begin{align*}
W_M&:=\frac{1}{M\sum_{i=1}^Mu_{i}^2}\sum\limits_{k=1}^M\sum\limits_{i=1}^Mu_{i}^2E_{ik}^2,\\
X_M&:=\frac{1}{M\sum_{i=1}^Mu_{i}^2}\sum\limits_{k=1}^M\sum\limits_{i=1}^M\sum\limits_{j=1}^{i-1}u_{i}u_jE_{ik}E_{jk}.
\end{align*}
First we show that $W_M$ converges to $\sigma^2$ almost surely. This part is an adaption of some arguments of the classical strong law of large numbers (cf. \cite{Etemadi1981}). Here we do not even need truncation arguments, since the entries $E_{ij}$ have finite fourth moments. By the Borel-Cantelli Lemma we get for $k_n:=\lfloor n^\beta \rfloor,~n\in\N$ that $(W_{k_n})$ converges to $\sigma^2$ almost surely as Chebyshev's inequality yields for any $\varepsilon>0$
\begin{align*}
\sum_{n=1}^\infty\P\left(\left|W_{k_n}-\sigma^2\right|>\varepsilon\right)&\le \sum_{n=1}^\infty\frac{\Var(W_{k_n})}{\varepsilon^2}\\
&\le\sum_{n=1}^\infty\varepsilon^{-2}k_n^{-2}\left(\sum_{i=1}^{k_n} u_i^2\right)^{-2}\sum_{k=1}^{k_n}\sum_{i=1}^{k_n}u_i^4\Var(E_{ik}^2)\\
&\le \frac{\E E_{11}^4}{\varepsilon^2}\sum_{n=1}^\infty k_n^{-1}<\infty.
\end{align*} 
For $M\in\N$ pick $n\in\N$ such that $k_n< M\le k_{n+1}$. This implies
\begin{align*}
k_{n}\ge \lfloor(M^\frac{1}{\beta}-1)^{\beta}\rfloor \text{ and }M> \lfloor(k_{n+1}^\frac{1}{\beta}-1)^{\beta}\rfloor.
\end{align*}
Now by monotonicity of $(M\sum_{i=1}^Mu_{i}^2W_M)$ and condition (\ref{eqn:SLL2}) follows
\begin{align*}
\sigma^2&\le \underset{M\to\infty}{\lim\inf}~\frac{k_n\sum_{i=1}^{k_n}u_i^2}{M\sum_{i=1}^Mu_i^2}W_{k_n}
\le\underset{M\to\infty}{\lim\inf}~W_M\\&\le \underset{M\to\infty}{\lim\sup}~W_M
\le\underset{M\to\infty}{\lim\sup}\frac{k_{n+1}\sum_{i=1}^{k_{n+1}}u_i^2}{M\sum_{i=1}^Mu_i^2}  W_{k_{n+1}} \le \sigma^2\text{ a.s.}
\end{align*}
So, $(W_M)$ converges to $\sigma^2$ almost surely.\\
Now consider $(X_M)$. Let $(k_n)$ be as before. Then again by the Borel-Cantelli Lemma $(X_{k_n})$ converges almost surely to $0$. For any $M\in\N$ pick $n\in\N$ again such that $k_n< M\le k_{n+1}$. We have
\begin{align*}
X_M&= \frac{1}{M\sum_{i=1}^Mu_{i}^2}\sum\limits_{k=1}^{k_n}\sum\limits_{i=1}^{k_n}\sum\limits_{j=1}^{i-1}u_{i}u_jE_{ik}E_{jk}\\
	&\hspace{1cm}		+\frac{1}{M\sum_{i=1}^Mu_{i}^2}\sum\limits_{k=1}^{k_n}\sum\limits_{i=k_n+1}^{M}\sum\limits_{j=1}^{i-1}u_{i}u_jE_{ik}E_{jk} \\
	&\hspace{1cm}		+\frac{1}{M\sum_{i=1}^Mu_{i}^2}\sum\limits_{k=k_n+1}^M\sum\limits_{i=1}^M\sum\limits_{j=1}^{i-1}u_{i}u_jE_{ik}E_{jk}
\end{align*}
Clearly, the first and the last term go to zero almost surely. It remains to prove that 
\begin{align*}
V_M:=\frac{1}{M\sum_{i=1}^Mu_{i}^2}\sum\limits_{k=1}^{k_n}\sum\limits_{i=k_n+1}^{M}\sum\limits_{j=1}^{i-1}u_{i}u_jE_{ik}E_{jk}\to 0\text{ a.s.}
\end{align*}
Therefore we estimate $\text{Var}(V_M)$:
\begin{align*}
\text{Var}(V_M)&\le \frac{\sigma^4}{M\left(\sum_{i=1}^Mu_{i}^2\right)^2}\sum\limits_{i=k_n+1}^Mu_{i}^2\cdot \sum\limits_{i=1}^{k_n}u_{i}^2\\
&\le \frac{\sigma^4}{M}\cdot \frac{\sum_{i=k_n+1}^Mu_{i}^2}{\sum_{i=1}^Mu_{i}^2}\\
&\le \frac{\sigma^4}{M}\cdot \frac{\sum_{i=B}^Mu_{i}^2}{\sum_{i=1}^Mu_{i}^2}\\
&\le \frac{c\sigma^4}{M^{1+\beta'}}.
\end{align*}
By the Borel-Cantelli Lemma and Chebyshev's inequality again we conclude the desired claim.
\end{proof}
\noindent If only finitely many entries $u_i$ are non-zero, then the almost sure convergence follows directly from the classical strong law of large numbers, since 
\begin{align*}
Z_M=\frac{1}{M}\sum\limits_{k=1}^M\left(\sum\limits_{i=1}^M \frac{u_{i}}{\sqrt{\sum_{i=1}^Mu_{i}^2}} E_{ik}  \right)^2
\end{align*}
and the summands 
\begin{align*}
\left(\sum\limits_{i=1}^M \frac{u_{i}}{\sqrt{\sum_{i=1}^Mu_{i}^2}} E_{ik}  \right)^2,~M\ge M_0,
\end{align*}
are i.i.d. for $M_0$ large enough.

\subsection*{Acknowledgments}
This work was supported by the {\it Deutsche Forschungsgemeinschaft} research unit 1735, Ro 3766/3-1.\newline
I am grateful to my Ph.D. advisor, Angelika Rohde, for her encouragement and bringing this topic to my attention.

\printbibliography[heading=head]
\vspace{0.5cm}
\tiny
\noindent RUHR-UNIVERSIT\"AT BOCHUM, FAKULT\"AT F\"UR MATHEMATIK, 44780 BOCHUM, GERMANY \\
\noindent E-MAIL: kamil.jurczak@ruhr-uni-bochum.de

\end{document}                                                                                                                                                                                                                                                                                                                                                                                                                                                                                                                                                                                                                                                                                                                                                                                                                                                                                                                                                                                                                                                                                                                                                                                                                                                                                                                                                                                                                                                                                                                                                                                                                                                                                                                                                                                                                                                                                                                                                                                                                                                                                                                                                                                                                                                                                                                                                                                                                                                                                                                                                                                                                                                                                                                                                                                                                                                                                                                                                                                                                                                                                                                                                                                                                                                                                                                                                                                                                                                                                                                                                                                                                                                                                                                                                                                                                                                                                                                                                                                                                                                                                                                                                                                                                                                                                                                                                                                                                                                                                                                                                                                                                                                                                                                                                                                                                                                                                                                                                                                                                                                                                                                                                                                                                                                                                                                                                                                                                                                                                                                                                                                                                                                                                                                                                                                                                                                                                                                                                                                                                                                                                                                                                                                                                                                                                                                                                                                                                                                                                                                                                                                                                                                                                                                                                                                                                                                                                                                                                                                                                                                                                                                                                                                                                                                                                                                                                                                                                                                                                                                                                                                                                                                                                                                                                                                                                                                                                                                                                                                                                                                                                                                                                                                                                                                                                                                                                                                                                                                                                                                                                                                                                                                                                                                                                                                                                                                                                                                                                                                                                                                                                                                                                                                                                                                                                                                                        